\documentclass[12pt,reqno]{amsart}
\usepackage{amsfonts}
\usepackage{amssymb, amsmath, latexsym}
\usepackage{longtable}
\usepackage{lipsum}
\usepackage{amsthm}
\usepackage{array}
\usepackage{hyperref}
\usepackage{mleftright}
\usepackage{orcidlink}
\usepackage[numbers,sort&compress]{natbib}
\setcitestyle{square}
\usepackage{graphicx,mathtools}
\usepackage{mathtools}

\usepackage{array}
\usepackage{graphicx,mathtools}
\usepackage{ragged2e}

\numberwithin{equation}{section}
\theoremstyle{plain}
\newtheorem{theorem}{Theorem}[section]

\newtheorem{definition}{Definition}[section]
\newtheorem{lemma}[theorem]{Lemma}
\newtheorem{example}{Example}[section] 
\newtheorem{conjecture}[theorem]{Conjecture}

\newtheorem{proposition}[theorem]{Proposition}
\newtheorem{remark}[theorem]{Remark}

\allowdisplaybreaks
\title{Difference of the Sum of All Odd and Even Overlined Parts of Overpartitions}
\author[N. D. Baruah]{NAYANDEEP DEKA BARUAH\, \orcidlink{0000-0003-1129-2929}}
\address{Department of Mathematical Sciences, Tezpur University,  Assam 784028, India}
\email{nayan@tezu.ernet.in, nayandeeptezu@gmail.com}
\author[P. GOGOI]{Pankaj Gogoi\, \orcidlink{0009-0002-4155-5654}}
\address{Department of Mathematical Sciences, Tezpur University,  Assam 784028, India}
\email{msp23110@tezu.ac.in, gopankajgo07@gmail.com}
\date{}


\begin{document}

\begin{abstract}
Recently, Garvan and Sarma \cite{GarvanSarma2026}  studied sums of non-overlined parts in overpartitions and partitions without repeated odd parts. In this paper, we study the corresponding statistic for overlined parts. We define $\textup{OSOME}(n)$ as the sum of all odd overlined parts minus the sum of all even overlined parts in the overpartitions of $n$. We obtain a closed form for its generating function and use it to prove congruences by elementary $q$-series methods, the theory of half-integral weight modular forms, and known congruences for the overpartition function.
\end{abstract}

\maketitle

\section{Introduction}

For complex numbers $a$ and $q$ with $|q|<1$, define
$$
 (a;q)_0:=1,\qquad
 (a;q)_n:=\prod_{j=0}^{n-1}(1-aq^j),\qquad
 (a;q)_\infty:=\prod_{j=0}^{\infty}(1-aq^j).
$$
Throughout the paper, we write
\begin{align*}
    f_r:&=(q^r;q^r)_\infty\qquad (r\geq1).
\end{align*}

A partition of a nonnegative integer $n$ is a non-increasing sequence of positive integers whose sum is $n$. Let $p(n)$ denote the number of partitions of $n$. Euler proved that
$$
 \sum_{n=0}^{\infty}p(n)q^n=\frac1{f_1}.
$$
The following congruences given by Ramanujan \cite{Ramanujan1919} are among the most well-known arithmetic properties satisfied by $p(n)$. For all nonengative integers $n$,
\begin{align*}
 p(5n+4)&\equiv0\pmod5,\\
 p(7n+5)&\equiv0\pmod7,\\
 p(11n+6)&\equiv0\pmod{11}.
\end{align*}

Recently, Andrews and Ghosh Dastidar \cite{AndrewsGhoshDastidar2026} introduced the functions $\textup{SOME}(n)$ and $\textup{DSOME}(n)$ related to partitions. Here $\textup{SOME}(n)$ is the sum of all the odd parts minus the sum of all the even parts in the partitions of $n$, and $\textup{DSOME}(n)$ is the analogous difference for partitions into distinct parts. The current authors \cite{BaruahGogoi2026} obtained a closed form for the generating function of $\textup{DSOME}(n)$ and used it to derive new congruences modulo $4$ and $8$.

An overpartition \cite{CorteelLovejoy2004} is a partition in which the first occurrence (equivalently, the final occurrence) of a number may be overlined. Let $\overline{p}(n)$ denote the number of overpartitions of $n$. Its generating function is
\begin{equation}\label{overpartition-gf}
 \sum_{n=0}^{\infty}\overline{p}(n)q^n
 =\frac{(-q;q)_\infty}{(q;q)_\infty}
 =\frac{f_2}{f_1^2}
 =\frac1{\varphi(-q)},
\end{equation}
where
$$
 \varphi(q):=\sum_{n=-\infty}^{\infty}q^{n^2}.
$$
\begin{example}\label{overpartititonexample}
The overpartitions of $4$ are $4$, $\overline4$, $3+1$, $\overline3+1$, $3+\overline1$, $\overline3+\overline1$, $2+2$,  $\overline2+2$, $2+1+1$, $\overline2+1+1$, $2+\overline1+1$,
 $\overline2+\overline1+1$, $1+1+1+1$,  and $\overline1+1+1+1$.
\end{example}
Congruences for the overpartition function have been studied by many authors; see, for example, Chen, Sun, Wang, and Zhang \cite{ChenSunWangZhang2015} and the general modular-form framework of Treneer \cite{Treneer2006}.

Recently, Gireesh and Hemanthkumar \cite{GireeshHemanthkumar2026}
introduced an overpartition analogue $\overline{\textup{SOME}}(n)$ of $\textup{SOME}(n)$. They established several congruences modulo $3$, $5$, and certain powers of $2$ for this function.

Very recently, Garvan and Sarma \cite{GarvanSarma2026} investigated the sums of non-overlined parts in overpartitions, alongside related statistics for partitions without repeated odd parts. For a given integer $n$, they considered the functions $\textup{SONO}(n)$, $\textup{SENO}(n)$,  and $\textup{SNO}(n)$ as the sum of all odd non-overlined parts, the sum of all  even non-overlined parts, and the total sum of all non-overlined parts, respectively, across all overpartitions of $n$. Utilizing classical theta-function identities and dissections, they established the following congruences:$$\begin{aligned} \mathrm{SONO}(5n + 3) &\equiv 0 \pmod{5},\\\mathrm{SENO}(7n + 5) &\equiv 0 \pmod{7}, \\  \mathrm{SNO}(5n + 2) &\equiv 0 \pmod{5}, \\ \mathrm{SNO}(5n + 4) &\equiv 0 \pmod{5}, \\ \mathrm{SNO}(7n + 3) &\equiv 0 \pmod{7}. \end{aligned}$$Motivated by their work, in this paper we study  the difference of the sum of all the odd and even overlined parts of overpartitions. Let $\textup{OSOME}_o(n)$ and $\textup{OSOME}_e(n)$ denote the sum of all odd and  even overlined parts, respectively, in all overpartitions of $n$. Define
\begin{equation}\label{OSOME-definition}
 \textup{OSOME}(n):=\textup{OSOME}_o(n)-\textup{OSOME}_e(n).
\end{equation}
It follows from Example \eqref{overpartititonexample} that 
 $\textup{OSOME}_o(4)=11$ and $\textup{OSOME}_e(4)=10$. Thus, $\textup{OSOME}(4)=1$.

Simple closed forms for the generating functions of $\textup{OSOME}(n)$, $\textup{OSOME}_o(n)$, and $\textup{OSOME}_e(n)$ are presented in the following theorem.
\begin{theorem}\label{CF}
We have
\begin{align}
 \sum_{n=0}^{\infty}\textup{OSOME}(n)q^n
 &=\frac18\left(\frac{f_2}{f_1^2}-\frac{f_1^6}{f_2^3}\right)
 =\frac18\left(\frac{1}{\varphi(-q)}-\varphi^3(-q)\right),\label{closedform}\\
 \sum_{n=0}^{\infty}\textup{OSOME}_o(n)q^n&=\frac{1}{24\varphi(-q)}\left(\varphi^4(q)-2\varphi^4(-q)+1\right)\label{closedform1},\\
 \sum_{n=0}^{\infty}\textup{OSOME}_e(n)q^n&=\frac{1}{24\varphi(-q)}\left(\varphi^4(q)+\varphi^4(-q)-2\right)\label{closedform2}.
\end{align}
\end{theorem}
The $\textup{OSOME}$ function satisfies Ramanujan-like congruences for small primes
\begin{theorem}\label{cong_for3,5,7,11}
For every nonnegative integer $n$, we have
\begin{align}
 \label{cong-mod3}\textup{OSOME}(24n+19)&\equiv0\pmod3,\\
  \label{cong-mod5}\textup{OSOME}(40n+27)&\equiv0\pmod5,\\
  \label{cong-mod8}\textup{OSOME}(8n+7)&\equiv0\pmod8,\\
  \label{cong-mod11}\textup{OSOME}(88n+55)&\equiv0\pmod{11}.
\end{align}
\end{theorem}
\begin{theorem}\label{mod3power}
For every nonnegative integer $n$ and $k$, we have
$$
 \textup{OSOME}\left(3^{2k+1}n+2\cdot3^{2k}\right)\equiv 0\pmod3,
$$
\end{theorem}

Shomanov and Garvan \cite{ShomanovGarvan2025} proved the following infinite family of Hecke-like congruences modulo powers of 2 for the overpartition function $\overline{p}(n)$. 
\begin{theorem}\label{Shomanov-Garvan-thmA}
Let $\ell$ be an odd prime that satisfies $\ell\equiv-1\pmod{2^{\alpha+12}},$ for some $\alpha\geq0$.
Then, for every positive integer $m$ prime to $\ell$,
$$
 \overline{p}\left(2^\alpha\ell^3m\right)\equiv0\pmod{2^{\alpha+12}}.
$$
\end{theorem}

In a similar fashion, we obtain the following infinite family of congruences modulo powers of $2$ for $\textup{OSOME}(n)$.
\begin{theorem}\label{mod2power}
Let $\alpha\geq0$, and let $\ell$ be an odd prime that satisfies 
$$
 \ell\equiv-1\pmod{2^{\alpha+12}}.
$$
Then, for every positive integer $m$ prime to $\ell$,
$$
 \textup{OSOME}\left(2^{\alpha}\ell^3m\right)\equiv0\pmod{2^{\alpha+9}}.
$$
\end{theorem}

Ryan, Sirolli, Villegas-Morales, and Zheng \cite{RyanSirolliVillegasZheng2024} gave infinite families of congruences modulo $m = 3,5,7$, and $11$ for the overpartition function $\overline{p}(n)$ as stated in the following two theorems. 

\begin{theorem}\label{Ryan-overpartition-lemma}For $m\in\{3,5,7,11\}$, define
\begin{equation}\label{km-definition}
 k_m:=
 \begin{cases}
  m+2,&m=3,\\
  m-2,&m>3.
 \end{cases}
\end{equation}
If $\ell$ is an odd prime such that $\ell^{k_m-2}\equiv-1\pmod m$, then for every positive integer $n$ prime to $\ell$, one has 
$$\overline{p}(m\ell^3n)\equiv0\pmod m.$$
\end{theorem}

\begin{theorem}
Let $m$ and $k_m$ be as given in the above theorem. Let $\ell$ be an odd prime such that
\begin{equation}\label{Ryan-square-condition}
 \ell^{k_m-2}\equiv-1+\varepsilon_{m,\ell}\ell^{(k_m-3)/2}\pmod m,
\end{equation}
where $\varepsilon_{m,\ell}\in\{\pm1\}$. Then for every positive integer $n$ prime to $\ell$ such that
\begin{equation}\label{Ryan-symbol-condition}
 \left(\frac{(-1)^{(k_m-1)/2}n}{\ell}\right)=\varepsilon_{m,\ell},
\end{equation}
one has
$$
 \overline{p}(m\ell^2n)\equiv0\pmod m.
$$
\end{theorem}

In the following two theorems, we present the corresponding results for $\textup{OSOME}(n)$. 
\begin{theorem}\label{general-cubic}
Let $m\in\{3,5,7,11\}$ and $\ell$ be an odd prime such that $\ell\equiv-1\pmod m$.
Then for any positive integer $n$ prime to $\ell$, one has
$$
 \textup{OSOME}(m\ell^3n)\equiv0\pmod m.
$$
\end{theorem}
\begin{theorem}\label{general-square} Let $m\in\{3,5,7,11\}$ and $k_m$ be as defined in \eqref{km-definition}. Let $\ell$ be an odd prime that satisfies 
\eqref{Ryan-square-condition}. Let $n$ be a positive integer prime to $\ell$ that also satisfies \eqref{Ryan-symbol-condition} and \begin{equation}\label{mod-eight-condition}
 n\equiv7m\pmod8.
\end{equation}
Then, for every $j\geq0$,
$$
 \textup{OSOME}(m4^j\ell^2n)\equiv0\pmod m.
$$
\end{theorem}
We present some preliminary results and some lemmas in the next section, whereas the final sections are devoted to proving the theorems.
\section{Preliminaries}


\subsection{Rogers -Ramanujan functions}
The Rogers–Ramanujan identities are
\begin{align}
G(q):&=\sum_{n=0}^{\infty}\frac{q^{n^2}}{(q;q)_n}=\frac{1}{(q;\;q^5)_\infty(q^4;\;q^5)_\infty}\label{Gq}\\
\intertext{and}    H(q):&=\sum_{n=0}^{\infty}\frac{q^{n^2+n}}{(q;q)_n}=\frac{1}{(q^2;\;q^5)_\infty(q^3;\;q^5)_\infty},\label{Hq}\end{align}
where $G(q)$ and $H(q)$ are called the Rogers--Ramanujan functions. It is to be noted that
\begin{align}\label{ghf5byf1}
G(q)H(q)=\frac{1}{(q;\;q^5)_\infty(q^2;\;q^5)_\infty(q^3;\;q^5)_\infty(q^4;\;q^5)_\infty}=\frac{f_5}{f_1}.
    \end{align}
Furthermore, by manipulating $q$-products, it is easy to see that
$$G(-q) = \dfrac{G(q^2)H^2(q^2)}{G(q)H(q^4)}\quad \textup{and}\quad H(-q) = \dfrac{H(q^2)G^2(q^2)}{H(q)G(q^4)}.$$ 
We will use the above identities without further reference to them. 
     
We will also make use of the following identities due to Son \cite{son}.
\begin{lemma}\label{sonlemma}
We have
    \begin{align}
        \varphi(q)+\varphi(q^{5})&=2f_2G(q)G(q^4)\label{SON1},\\
        \psi(q^2)+q\psi(q^{10})
        &=\frac{f_{10}G(q)G(q^2)}{H(q)G(q^4)}
        =f_2G(q)H(-q)\notag\\
        &=\frac{f_5f_{10}}{f_1}\frac{G(q^2)}{H^2(q)G(q^4)},\label{SON2}        
    \end{align}
   where 
   \begin{align*}
\psi(q):=\sum_{n=0}^{\infty}q^{n(n+1)/2}.
\end{align*}
\end{lemma} 

In the next four lemmas, we present some more useful identities.
\begin{lemma}\label{lemma3}
    We have
    \begin{align}
G^2(q^4)H(-q)
&=\frac{f_1f_{10}}{f_2^2f_5}G(q^2)
\left(\psi(q^2)+q\psi(q^{10})\right)
\label{eq:G4H2}\\\intertext{and}
G^2(q^4)H(-q)
&=\frac{f_1f_{10}}{2f_2^2f_5}G(q^2)
\left(\varphi(q)+\varphi(q^5)\right).
\label{eq:G42H}
\end{align}
\end{lemma}
\begin{proof}We have
    \begin{align*}
G(q^4)H^2(-q)
&=H(-q)\left(G(q^4)H(-q)\right)
\\
&=H(-q)\frac{G^2(q^2)H(q^2)}{H(q)}
\\
&=\frac{f_{10}}{f_2}\frac{G(q^2)}{H(q)}H(-q)
\\
&=\frac{f_1f_{10}}{f_2f_5}G(q^2)G(q)H(-q),\end{align*}
which, by \eqref{SON2}, gives
\begin{align*}G(q^4)H^2(-q)&=\frac{f_1f_{10}}{f_2^2f_5}G(q^2)
\left(\psi(q^2)+q\psi(q^{10})\right),
\end{align*}
which proves the first identity of the lemma.

Similarly,
\begin{align*}
G^2(q^4)H(-q)
&=G(q^4)\left(G(q^4)H(-q)\right)
\notag\\
&=G(q^4)\frac{G^2(q^2)H(q^2)}{H(q)}
\notag\\
&=\frac{f_{10}}{f_2}\frac{G(q^2)G(q^4)}{H(q)}
\notag\\
&=\frac{f_1f_{10}}{f_2f_5}G(q)G(q^2)G(q^4),
\end{align*}
which, by  \eqref{SON1}, implies that
\begin{align*}
G^2(q^4)H(-q)&=\frac{f_1f_{10}}{2f_2^2f_5}G(q^2)
\left(\varphi(q)+\varphi(q^5)\right).
\end{align*}
This proves the second identity to complete the proof.
\end{proof}

\begin{lemma}\label{lemma=0}Set
\begin{align}
 \Phi&:=\varphi(q^2)+\varphi(q^{10})\quad and\quad
 \Psi:=\psi(q^4)+q^2\psi(q^{20}).
\label{PhiPsi}
\end{align}We have
\begin{align}\label{Phi-0}
     f_4\left(2\Psi G(-q^2)-q\Phi H(q^8)\right)
-2f_1\psi(q^5)G(q^2)=0.
\end{align}
\end{lemma}
\begin{proof} Using \eqref{SON1} and \eqref{SON2}, we have 
    \begin{align*}
&f_4\left(2\Psi G(-q^2)-q\Phi H(q^8)\right)
-2f_1\psi(q^5)G(q^2)\notag\\
&=\frac{2f_4f_{20}G(q^2)G(q^4)G(-q^2)}{H(q^2)G(q^8)}
-2qf_4^{\,2}G(q^2)G(q^8)H(q^8)
-\frac{2f_1f_{10}^{\,2}G(q^2)}{f_5}\notag\\
&=2G(q^2)\left(
f_4f_{20}\frac{G(q^4)G(-q^2)}{H(q^2)G(q^8)}
-qf_4^2G(q^8)H(q^8)
-\frac{f_1f_{10}^{\,2}}{f_5}
\right)\notag\\
&=2G(q^2)\left(f_4f_{20}\frac{G(q^4)}{H(q^2)G(q^8)}\cdot
  \frac{G(q^4)H^2(q^4)}{G(q^2)H(q^8)}
-qf_4^2G(q^8)H(q^8)
-\frac{f_1f_{10}^{\,2}}{f_5}
\right)\notag\\
&=2G(q^2)\left(\frac{f_2f_8f_{20}^{\,3}}{f_4f_{10}f_{40}}
-q\frac{f_4^2f_{40}}{f_8}
-\frac{f_1f_{10}^{\,2}}{f_5}
\right)\\
&=0,
\end{align*}
where we applied \eqref{2disf_1/f5} to arrive at the last equality.
\end{proof}


\begin{lemma}\label{lemma2}
We have
\begin{align}
f_1f_2^2f_4^2f_8^2G(q^8)
+\frac{2f_1f_4^2f_{20}^2G(q^4)}{f_8^2H(q^2)G(q^8)}\equiv\frac{f_1f_2^8f_4^2}{2G(q^2)}\pmod5.\label{Identity_1}
\end{align}
\end{lemma}
    \begin{proof}
        From \eqref{SON1} and \eqref{SON2}, we have
        \begin{align}
\varphi(q^2)+\varphi(q^{10})&=2f_4G(q^2)G(q^8),
\label{eq:D5phi}\\
\psi(q^4)+q^2\psi(q^{20})
&=\frac{f_{20}G(q^2)G(q^4)}{H(q^2)G(q^8)}.
\label{eq:D5psi}
\end{align}
Using \eqref{eq:D5phi}, the first term in \eqref{Identity_1} becomes
\begin{align}
f_1f_2^2f_4^2f_8^2G(q^8)
&=\frac{f_1f_4}{2G(q^2)}\,
f_2^2f_8^2(2f_4G(q^2)G(q^8))\notag\\
&=\frac{f_1f_4}{2G(q^2)}\,
f_2^2f_8^2\{\varphi(q^2)+\varphi(q^{10})\}.\label{first_term}
\end{align}
Similarly, \eqref{eq:D5psi} gives
\begin{align}
\frac{2f_1f_4^2f_{20}^2G(q^4)}{f_8^2H(q^2)G(q^8)}
&=\frac{f_1f_4}{2G(q^2)}\,
\frac{4f_4f_{20}}{f_8^2}
\frac{f_{20}G(q^2)G(q^4)}{H(q^2)G(q^8)}\notag\\
&=\frac{f_1f_4}{2G(q^2)}\,
\frac{4f_4f_{20}}{f_8^2}
\{\psi(q^4)+q^2\psi(q^{20})\}.\label{second_term}
\end{align}
Combining \eqref{first_term} and \eqref{second_term},
 we obtain
\begin{align}
&f_1f_2^2f_4^2f_8^2G(q^8)
+\frac{2f_1f_4^2f_{20}^2G(q^4)}{f_8^2H(q^2)G(q^8)}\notag\\
&\equiv\frac{f_1f_4}{2G(q^2)}
\left(
f_2^2f_8^2\{\varphi(q^2)+\varphi(q^{10})\}
+\frac{4f_4f_{20}}{f_8^2}
\left(\psi(q^4)+q^2\psi(q^{20})\right)
\right)\pmod5\notag\\
&\equiv\frac{f_1f_4}{2G(q^2)}\left(f_2^2f_8^2\left(\varphi(q^2)+\varphi^5(q^2)\right)+\frac{4f_4^6}{f_8^2}
\left(\psi(q^4)+q^2\psi^5(q^4)\right)
\right)\pmod5.\label{reduction_1}
\end{align}
We also have
\begin{equation}
\varphi(q^2)=\frac{f_4^5}{f_2^2f_8^2},
\qquad
\varphi(-q^2)=\frac{f_2^2}{f_4},
\qquad
\psi(q^4)=\frac{f_8^2}{f_4}.
\label{eq:shifted-theta-products}
\end{equation}
Substituting \eqref{eq:shifted-theta-products} into \eqref{reduction_1}, we get
\begin{align}
f_1f_2^2f_4^2f_8^2G(q^8)
+\frac{2f_1f_4^2f_{20}^2G(q^4)}{f_8^2H(q^2)G(q^8)}
&\equiv\frac{f_1f_4}{2G(q^2)}
\left(f_4^5
\left(1+\varphi^4(q^2)+4+4q^2\psi^4(q^4)\right)
\right)\notag\\
&\equiv\frac{f_1f_4^6}{2G(q^2)}
\left(\varphi^4(q^2)-q^2\psi^4(q^4)\right)
\pmod5.\label{reduction_2}
\end{align}
By \cite[Entry 25, p.~40]{RN3},
\begin{align}
    \varphi^4(q)-\varphi^4(-q)
    =16q\psi^4(q^2)&\equiv q\psi^4(q^2)\pmod5,
    \intertext{replacing $q$ by $q^2$ on the above gives}
    \varphi^4(q^2)-q^2\psi^4(q^4)&\equiv\varphi^4(-q^2)\pmod5.\label{identity_3}
\end{align}
Using \eqref{identity_3} in \eqref{reduction_2}, we finally obtain
\begin{align}
f_1f_2^2f_4^2f_8^2G(q^8)
+\frac{2f_1f_4^2f_{20}^2G(q^4)}{f_8^2H(q^2)G(q^8)}
&\equiv\frac{f_1f_4^6}{2G(q^2)}\frac{f_2^8}{f_4^4}\equiv\frac{f_1f_2^8f_4^2}{2G(q^2)}\notag.
\end{align}
This completes the proof.
\end{proof}
\begin{lemma}\label{lemmaforreduction}
We have
\begin{align}
2G(-q^5)+qH(-q^5)
&\equiv
\frac{2f_1^2f_4^4}{f_2f_{10}}
\left(G(q^{20})+2q^4H(q^{20})\right)
\pmod5.\label{Identity_1mod5}
    \end{align}
\end{lemma}
\begin{proof}
    From \cite[Entry 8.3.1, p.~222]{AndrewsBerndtIII}, we have
    \begin{align}
        H(q)G^{11}(q)-q^2G(q)H^{11}(q)
=1+11qG^6(q)H^6(q).\label{eq:ramanujanLNB}
    \end{align}
    Replacing $q$ by $q^4$ in the above, it follows that
    \begin{align*}
        G^{10}(q^4)-q^8H^{10}(q^4)
=&\frac{1}{G(q^4)H(q^4)}
   +11q^4G^5(q^4)H^5(q^4).\end{align*}
  Using \eqref{ghf5byf1} in the above, we find that 
  \begin{align} 
G^{10}(q^4)
-q^4G^5(q^4)H^5(q^4)
-q^8H^{10}(q^4)&\equiv\frac{1}{G(q^4)H(q^4)}\equiv\frac{f_4}{f_{20}}\pmod{5}.\label{GH_relation_1} 
    \end{align}

We also have
\begin{align}
    \left(G(q^{20})+2q^4H(q^{20})\right)^2
&\equiv\left( G^5(q^4)+2q^4H^5(q^4)\right)^2\notag\\
&\equiv G^{10}(q^4)
-q^4G(q^4)^5H^5(q^4)
-q^8H^{10}(q^4)\pmod5.\label{GH_relation_2}
\end{align}
From \eqref{GH_relation_1} and \eqref{GH_relation_2}, we have
\begin{align}
    \left(G(q^{20})+2q^4H(q^{20})\right)^{2}
\equiv\frac{f_4}{f_{20}}
\pmod5.\label{GH_relation_3}
\end{align}

Next, replacing $q$ by $-q$ in \eqref{eq:ramanujanLNB} and then dividing by
$G(-q)H(-q)$, we see that
\begin{align}
   \frac{1}{G(-q)H(-q)}
\equiv{}&
G^{10}(-q)
+qG^5(-q)H^5(-q)
-q^2H^{10}(-q)
\pmod5.\label{GH_relation_4}
\end{align}

Now, 
\begin{align*}
2G^5(-q)+qH^5(-q)
\equiv 2G(-q^5)+qH(-q^5)\pmod{5}.
\end{align*}
Therefore,
\begin{align}
&\left(2G(-q^5)+qH(-q^5)\right)^2\equiv -G^{10}(-q)
-qG^5(-q)H^5(-q)
+q^2H^{10}(-q)\pmod5.\label{GH_relation_5}
\end{align}
From \eqref{GH_relation_4} and \eqref{GH_relation_5}, it follows that
\begin{align}
    \left(2G(-q^5)+qH(-q^5)\right)^{2}
\equiv-\frac{1}{G(-q)H(-q)}
\pmod5.\label{GH_relation_6}
\end{align}

Replacing $q$ by $-q$ in \eqref{ghf5byf1}, we find that
\begin{align*}
  G(-q)H(-q)=\frac{f_1f_4f_{10}^3}{f_2^3f_5f_{20}}
\end{align*}
Therefore, \eqref{GH_relation_6} reduces to
\begin{align}
    \left(2G(-q^5)+qH(-q^5)\right)^{2}
\equiv-\frac{f_2^3f_5f_{20}}{f_1f_4f_{10}^3}
\pmod5.\label{GH_relation_66}
\end{align}
From \eqref{GH_relation_3} and \eqref{GH_relation_66},  we have
\begin{align}
    \left(
\frac{2G(-q^5)+qH(-q^5)}
     {G(q^{20})+2q^4H(q^{20})}
\right)^{2}
&\equiv-\frac{f_2^3f_5f_{20}^2}{f_1f_4^2f_{10}^3}\\
&\equiv\left(2\frac{f_1^2f_4^4}{f_2f_{10}}\right)^2
\pmod5.
\end{align}
Taking square roots and then comparing the constant terms, we arrive that
\begin{align*}2G(-q^5)+qH(-q^5)
&\equiv
\frac{2f_1^2f_4^4}{f_2f_{10}}
\left(G(q^{20})+2q^4H(q^{20})\right)
\pmod5.
\end{align*}
 This completes the proof of \eqref{Identity_1mod5}. 
 \end{proof}

 \subsection{\texorpdfstring{$t$}{t}-dissection of \texorpdfstring{$q$}{q}-series}
\begin{definition}
Let $t\geq2$ be an integer and let
$$
F(q)=\sum_{n\geq0}a(n)q^n.
$$
The $t$-dissection of $F(q)$ is the decomposition
$$
F(q)=\sum_{r=0}^{t-1}q^rF_r(q^t),\quad
\textup{where,}\quad
F_r(q)=\sum_{n\geq0}a(tn+r)q^n.
$$
Thus, $F_r(q)$ contains precisely the coefficients of $F(q)$ whose
indices are congruent to $r$ modulo $t$.
\end{definition}

First we recall 4-dissections of $f_2G(q)$ and $f_2H(q)$ due to Watson \cite[p.~60]{watson}.
\begin{lemma}
We have
\begin{align}    f_2G(q)&=f_8\left(G\left(q^{16}\right)+qH(-q^4)\right)\label{watson1}\\\intertext{and}
    f_2H(q)&=f_8\left(G(-q^{4})+q^3H(q^{16})\right)\label{watson2}.
\end{align}
\end{lemma}

In the following lemma, we state some well-known $2$-, $3$- and $5$- dissections.
\begin{lemma}We have
\begin{align}
    \frac{f_1^2}{f_2}&=\frac{f_8^5}{f_4^2f_{16}^2}-2q\frac{f_{16}^2}{f_8},\label{eq:two-dissection-one}\\
    \frac{f_2}{f_1^2}&=\frac{f_8^5}{f_2^4f_{16}^2}+2q\frac{f_4^2f_{16}^2}{f_2^4f_8},\label{2disreciprocal}\\
    f_1^4&=\frac{f_4^{10}}{f_2^2f_8^4}-4q\frac{f_2^2f_8^4}{f_4^2},\label{2disf14}\\
    \frac{1}{f_1^4}&=\frac{f_4^{14}}{f_2^{14}f_8^4}+4q\frac{f_4^2f_8^4}{f_2^{10}},\label{2dis1/f14}\\
    \frac{f_1}{f_5}
&=\frac{f_2f_8f_{20}^3}{f_4f_{10}^3f_{40}}
+q\frac{f_4^2f_{40}}{f_8f_{10}^2},\label{2disf_1/f5}\\
     \frac{f_1^2}{f_2}&=\frac{f_9^2}{f_{18}}-2q\frac{f_3f_{18}^2}{f_6f_9},\label{3disf_1f_2}\\  \varphi(q)&=\varphi(q^{25})+2qf(q^{15},q^{35})+2q^4f(q^5,q^{45}),\label{5-dissofvarphi}
\end{align}
where 
\begin{align*}
f(a,b)&:=\sum_{n=-\infty}^{\infty}a^{n(n+1)/2}b^{n(n-1)/2}.
\end{align*}
\end{lemma}

Finally we recall a well-known $5$-dissection of $1/f_1$.
\begin{lemma}\textnormal{\cite[p.~165]{spirit}} We have
    \begin{align}    
        \label{c2}\frac{1}{f_1}&=\frac{f_{25}^5}{f_5^6}\bigg(\dfrac{G^4(q^5)}{H^4(q^5)}+q\dfrac{G^3(q^5)}{H^3(q^5)}+2q^2\dfrac{G^2(q^5)}{H^2(q^5)}+3q^3\dfrac{G(q^5)}{H(q^5)}+5q^4\notag\\
        &\quad-3q^5\dfrac{H(q^5)}{G(q^5)}+2q^6\dfrac{H^2(q^5)}{G^2(q^5)}-q^7\dfrac{H^3(q^5)}{G^3(q^5)}+q^8\dfrac{H^4(q^5)}{G^4(q^5)}\bigg).
    \end{align}
\end{lemma}

\subsection{Half-integral weight modular forms}
We recall only the facts needed in this paper. Let
$$
 \Gamma_0(N):=
 \left\{
 \begin{pmatrix}a&b\\ c&d\end{pmatrix}\in\mathrm{SL}_2(\mathbb Z):
 c\equiv0\pmod N
 \right\}.
$$
For an odd positive integer $k$ and a positive integer $N$ divisible by $4$, write
$$
 M_{k/2}\bigl(\widetilde{\Gamma}_0(N)\bigr)
$$
for the space of holomorphic modular forms of weight $k/2$ on the metaplectic cover of $\Gamma_0(N)$. For detailed treatments, see Shimura \cite{Shimura1973}, Koblitz \cite{Koblitz1993}, Ono \cite{Ono2004}, and Wang and Pei \cite{WangPei2012}.

Let
$$
 f(z)=\sum_{n\geq0}a(n)q^n\in M_{k/2}\bigl(\widetilde{\Gamma}_0(N)\bigr),
 \qquad q=e^{2\pi iz},
$$
and let $\ell$ be an odd prime with $\ell\nmid N$. The half-integral weight Hecke operator is given by
\begin{align}
 f\big|T_{k/2,N}(\ell^2)
 =\sum_{n\geq0}\Bigg(&a(\ell^2n)
 +\left(\frac{(-1)^{(k-1)/2}n}{\ell}\right)
 \ell^{(k-3)/2}a(n)
 +\ell^{k-2}a\left(\frac{n}{\ell^2}\right)\Bigg)q^n,
\label{half-integral-Hecke}
\end{align}
where $a(n/\ell^2)=0$ if $\ell^2\nmid n$. Unlike the integral-weight case, the natural index is $\ell^2$.

The theta function $\varphi(q)$ is a modular form of weight $1/2$ on $\Gamma_0(4)$. Consequently,
$$
 \varphi^3(q)\in M_{3/2}\bigl(\widetilde{\Gamma}_0(4)\bigr).
$$
Write
\begin{equation}\label{phi-cube}
 \varphi^3(q)=\sum_{n\geq0}r_3(n)q^n,
\end{equation}
where $r_3(n)$ is the number of ordered triples $(x,y,z)\in\mathbb Z^3$ such that
$$
 x^2+y^2+z^2=n.
$$

We shall use the standard dimension formula
$$
 \dim M_{k/2}\bigl(\widetilde{\Gamma}_0(4)\bigr)
 =1+\left\lfloor\frac{k}{4}\right\rfloor
 \qquad(k\text{ odd}).
$$
For $k=3$, this space is one-dimensional. This gives the following Hecke relation.

\begin{proposition}\label{theta-Hecke-proposition}
For every odd prime $\ell$,
\begin{equation}\label{theta-Hecke-eigenvalue}
 \varphi^3(q)\big|T_{3/2,4}(\ell^2)=(\ell+1)\varphi^3(q).
\end{equation}
Consequently,
\begin{equation}\label{r3-Hecke-relation}
 r_3(\ell^2n)+\left(\frac{-n}{\ell}\right)r_3(n)
 +\ell r_3\left(\frac{n}{\ell^2}\right)
 =(\ell+1)r_3(n)
\end{equation}
\end{proposition}

\begin{proof}
The Hecke operator preserves $M_{3/2}(\widetilde{\Gamma}_0(4))$. Since this space is one-dimensional and is spanned by $\varphi(q)^3$, there is a number $\lambda(\ell)$ such that
$$
 \varphi^3(q)\big|T_{3/2,4}(\ell^2)=\lambda(\ell)\varphi(q)^3.
$$
The constant term of $\varphi^3(q)$ is $1$. In \eqref{half-integral-Hecke}, setting $k=3$ and $n=0$ shows that the constant term after applying the Hecke operator is
$$
 1+\ell.
$$
Hence $\lambda(\ell)=\ell+1$, proving \eqref{theta-Hecke-eigenvalue}. Comparing the coefficient of $q^n$ gives \eqref{r3-Hecke-relation}.
\end{proof}

Define
\begin{equation}\label{c-definition}
 \varphi^3(-q)=\sum_{n\geq0}c(n)q^n.
\end{equation}
Replacing $q$ by $-q$ in \eqref{phi-cube} gives
\begin{equation}\label{c-r3-relation}
 c(n)=(-1)^nr_3(n).
\end{equation}
Since $n$, $\ell^2n$, and $n/\ell^2$ have the same parity whenever the last term is defined, Proposition \ref{theta-Hecke-proposition} immediately yields the next lemma.

\begin{lemma}\label{c-Hecke-relation}
For every odd prime $\ell$ and every nonnegative integer $n$,
\begin{equation}\label{c-relation}
 c(\ell^2n)+\left(\frac{-n}{\ell}\right)c(n)
 +\ell c\left(\frac{n}{\ell^2}\right)
 =(\ell+1)c(n),
\end{equation}
where $c(x)=0$ when $x$ is not a nonnegative integer.
\end{lemma}

\section{Proof of Theorems \ref{CF}}
\noindent \emph{Proof of} \eqref{closedform} Since the generating function of the overpartition function $\overline{p}(n)$ is $(-q;q)_\infty/(q;q)_\infty$, where the numerator term is responsible for overlined parts, it follows that
         \begin{align*}
             \sum_{n=0}^{\infty}\textup{OSOME}(n)q^n&=\frac{\partial}{\partial z}\bigg|_{z=1}\prod_{n=1}^{\infty}\frac{\big(1+(qz)^{2n-1}\big)(1+(qz^{-1})^{2n})}{(q;q)_\infty}\notag\\
             &=\frac{(-q;q)_\infty}{(q;q)_\infty}\left(\sum_{n\geq0}\frac{(2n-1)q^{2n-1}}{1+q^{2n-1}}-\sum_{n\geq0}\frac{2nq^{2n}}{1+q^{2n}}\right).
         \end{align*}
         However, from \cite[p.~61, Eq. (3.3.6)]{spirit}, we recall that
 \begin{align*}
        \dfrac{(q;q)_\infty^4}{(-q;q)_\infty^4}=1-8\sum_{n=1}^{\infty}\left(\frac{(2n-1)q^{2n-1}}{1+q^{2n-1}}-\sum_{n\geq0}\frac{2nq^{2n}}{1+q^{2n}}\right).
  \end{align*}

  Hence, we find that
  \begin{align*}
        \sum_{n=0}^{\infty}\textup{OSOME}(n)q^n&=\frac{(-q;q)_\infty}{8(q;q)_\infty}\left(1-\frac{(-q;q)_\infty^4}{(q;q)_\infty^4}\right)\\
        &=\frac{1}{8}\left(\frac{f_2}{f_1^2}-\frac{f_1^6}{f_2^3}\right)\\
         &=\frac18\left(\frac{1}{\varphi(-q)}-\varphi(-q)^3\right).
    \end{align*}
    \emph{Proof of} \eqref{closedform1}
    We have 
    \begin{align}    \sum_{n=0}^{\infty}\textup{OSOME}_o(n)q^n&=\frac{\partial}{\partial z}\bigg|_{z=1}\prod_{n=1}^{\infty}\frac{\big(1+(qz)^{2n-1}\big)(1+q^{2n})}{(q;q)_\infty}\notag\\
             &=\frac{(-q;q)_\infty}{(q;q)_\infty}\left(\sum_{n\geq0}\frac{(2n-1)q^{2n-1}}{1+q^{2n-1}}\right)\label{OSOMEon}
    \end{align}
    However, from  \cite[p.~134, Eq. (6.2.3)]{spirit}, we recall that
    \begin{align*}
         \sum_{n\geq0}\frac{(2n-1)q^{2n-1}}{1-q^{2n-1}}=\frac{1}{24}\left(2\varphi(q)-\varphi^4(-q)-1\right),
    \intertext{replacing $q$ by $-q$ on the above, we have}
       \sum_{n\geq0}\frac{(2n-1)q^{2n-1}}{1+q^{2n-1}}=\frac{1}{24}\left(\varphi^4(q)-2\varphi(-q)+1\right).
   \end{align*}
Employing the above identity in \eqref{OSOMEon}, we find that
\begin{align*}    \sum_{n=0}^{\infty}\textup{OSOME}_o(n)q^n=\frac{1}{24\varphi^4(-q)}\left(\varphi^4(q)-2\varphi(-q)+1\right),
\end{align*}
which completes the proof of \eqref{closedform1}.

\noindent\emph{Proof of} \eqref{closedform2} The proof follows readily  from \eqref{closedform} and \eqref{closedform2}.
\section{Proofs of Theorems \ref{cong_for3,5,7,11}--\ref{mod3power}}\label{sec4}
  
The proof of \eqref{cong-mod5} is somewhat involved. So we first present the proofs of the remaining congruences.

\noindent \emph{Proofs of \eqref{cong-mod3},  \eqref{cong-mod8}, and  \eqref{cong-mod11}}.  From Theorem \ref{CF}, we note that
\begin{align*}
8\sum_{n\geq 0}\textup{OSOME}(n)q^n&=\frac{f_2}{f_1^2}-\frac{f_1^6}{f_2^3}.
\end{align*}
Therefore,
\begin{align*}
\sum_{n\geq 0}\textup{OSOME}(n)q^n&\equiv\frac{f_1^6}{f_2^3}-\frac{f_2}{f_1^2}\\
&\equiv\frac{f_3^2}{f_6}-\left(\frac{f_1^2}{f_2}\right)^2\frac{f_6}{f_3^2}\pmod{3},
\end{align*}
which by \eqref{3disf_1f_2} can be rewritten as 
\begin{align*}
\sum_{n\geq 0}\textup{OSOME}(n)q^n&\equiv\frac{f_3^2}{f_6}-\frac{f_6}{f_3^2}\left(\frac{f_9^2}{f_{18}}-2q\frac{f_3f_{18}^2}{f_6f_9}\right)^2\pmod{3}.
\end{align*}
Extracting the terms involving $q^{3n+1}$ from both sides of the above, dividing by $q$, and then replacing $q^3$ by $q$, we find that
\begin{align*}
\sum_{n\geq 0}\textup{OSOME}(3n+1)q^n
&\equiv\frac{f_3f_6}{f_1}\equiv
f_1^2f_2^3\pmod{3}.
\end{align*}
Employing \eqref{eq:two-dissection-one} and \eqref{2disf14} in the above, we have
\begin{align*}
\sum_{n\geq 0}\textup{OSOME}(3n+1)q^n
&\equiv f_1^2f_2^3\\
&\equiv \dfrac{f_1^2}{f_2}\cdot f_2^4\\
&\equiv\left(
\frac{f_8^5}{f_4^2f_{16}^2}
-
2q\frac{f_{16}^2}{f_8}
\right)
\left(
\frac{f_8^{10}}{f_4^2f_{16}^4}
-
4q^2\frac{f_4^2f_{16}^4}{f_8^2}
\right)\\
&\equiv \frac{f_8^{15}}{f_4^4f_{16}^6}
-
2q\frac{f_8^9}{f_4^2f_{16}^2}
-
4q^2f_8^3f_{16}^2
+
8q^3\frac{f_4^2f_{16}^6}{f_8^3}
\pmod{3}.
\end{align*}
Comparing the coefficients of $q^{8n+6}$, we obtain
\begin{align*}
\textup{OSOME}\bigl(3(8n+6)+1\bigr)\equiv 0\pmod{3},\end{align*}
which is equivalent to \eqref{cong-mod3}.

Next, employing \eqref{eq:two-dissection-one} and \eqref{2disreciprocal} in \eqref{closedform} and then extracting the terms involving the odd exponents of $q$, we find that
\begin{align*}
 \sum_{n\geq0}\textup{OSOME}(2n+1)q^n&=\frac{1}{4}\left(\frac{f_2^2f_8^2}{f_1^4f_4}+3\frac{f_4^9}{f_2^4f_8^2}+4q\frac{f_8^6}{f_4^3}
\right),
\end{align*}
which by \eqref{2dis1/f14} can be written as
\begin{align*}
 \sum_{n\geq0}\textup{OSOME}(2n+1)q^n&=\frac{1}{4}\left(\frac{f_8^2}{f_4}\left(\frac{f_{4}^{14}}{f_{2}^{12} f_{8}^{4}} + \frac{4 q f_{4}^{2} f_{8}^{4}}{f_{2}^{8}}\right)+3\frac{f_4^9}{f_2^4f_8^2}+4q\frac{f_8^6}{f_4^3}
\right).
\end{align*}
Extracting the terms involving $q^{2n+1}$ from both sides, we find that
\begin{align*}
    \sum_{n\geq0}\textup{OSOME}(4n+3)q^n&=\frac{f_2f_4^6}{f_1^8}+\frac{f_4^6}{f_2^3}.
\end{align*}
Employing \eqref{2dis1/f14} again in the above, we have
\begin{align*}
    \sum_{n\geq0}\textup{OSOME}(4n+3)q^n&=\frac{f_4^6}{f_2^3}+f_2f_4^6\left(\frac{f_4^{28}}{f_2^{28}f_8^8}+8q\frac{f_4^{16}}{f_2^{24}}+16q^2\frac{f_4^4f_8^8}{f_2^{20}}\right).
\end{align*}
Extracting the odd exponents of $q$ from both sides of the above, we find that
\begin{align}
   \sum_{n\geq0} \textup{OSOME}(8n+7)=8\frac{f_2^{22}}{f_1^{23}},\label{Gf8n+7}
\end{align}
from which \eqref{cong-mod8} follows readily.

Using $$
f_1^{22}\equiv f_{11}^{2} \pmod{11},
$$
 in \eqref{Gf8n+7}, we have
\begin{align}\label{gen-mod11}
\sum_{n\geq0}\textup{OSOME}(8n+7)q^n&\equiv8\frac{f_{22}^2}{f_{11}^2}\cdot\frac{1}{f_1}\\
&\equiv8\frac{f_{22}^2}{f_{11}^2}\sum_{n\geq0} p(n)q^n\pmod{11}.
\end{align}
Since $f_{22}^2/f_{11}^2$ is a series in powers of $q^{11}$ and
$$
p(11n+6)\equiv0\pmod{11},
$$
extracting the coefficients of $q^{11n+6}$ in \eqref{gen-mod11}, we arrive at
$$
\textup{OSOME}(88n+55)\equiv0\pmod{11}.
$$
This completes the proof of \eqref{cong-mod11}. 

We now proceed to prove \eqref{cong-mod5}. The following lemma simplifies the generating function for $\textup{OSOME}(n)$ and expresses it in a form suitable for the dissections.

\begin{lemma}\label{lemmabb}
\textnormal{(Baruah and Begum \cite[Eqs.~(2.6) and (2.7)]{ND})}
    \begin{align}
        \frac{f_5^5}{f_1^4f_{10}^3}&=\frac{f_5}{f_2^2f_{10}}+4q\frac{f_{10}^2}{f_1^3f_2}.\label{Lemma4a}       
     \end{align}
\end{lemma}

\noindent \emph{Proof of \eqref{cong-mod5}}.  From \eqref{closedform} and \eqref{Lemma4a},  we have
    \begin{align}
   \sum_{n=0}^{\infty}\textup{OSOME}(n)q^n
   &=\frac{1}{8}\left(\frac{f_1^2f_{10}^2}{f_2f_5^4}+4q\frac{f_{10}^5}{f_1f_5^5}\right)-\frac{1}{8}\left(\frac{f_1^2f_5^4}{f_2f_{10}^2}-4q\frac{f_1^3f_{10}^3}{f_2f_5}\right)\notag\\
   &=\frac{1}{8}\left(\frac{f_1^2f_{10}^2}{f_2f_5^4}+4q\frac{f_{10}^5}{f_1f_5^5}\right)-\frac{f_1^2f_5^4}{8f_2f_{10}^2}+\frac{q}{2}\left(\frac{f_5^3f_{10}}{f_1}-4q\frac{f_{10}^6}{f_2f_5^2}\right)\notag\\ 
    &=\frac{\varphi(-q)}{8\varphi^2(-q^5)}-\frac{1}{8}\varphi(-q)\varphi^2(-q^5)+\frac{q}{2}\frac{f_{10}^5}{f_1f_5^5}+\frac{q}{2}\frac{f_5^3f_{10}}{f_1}\notag\\
    &\quad-2q^2\frac{f_{10}^6}{f_2f_5^2}.\label{before_5_diss}
\end{align}
We intend to extract the terms containing $q^{5n+2}$ from both sides of the above. Note from \eqref{5-dissofvarphi} that the first two terms on the right side of the above do not contain terms of the form $q^{5n+2}$. Therefore, we need to extract the terms that contain $q^{5n+2}$ from the last three terms only. 

Let $\left[q^{5n+2}\right]\left\{F(q)\right\}$ denote
 the term after extracting the terms involving $q^{5n+2}$, dividing by $q^2$ and then replacing $q^5$ by $q$.

Employing \eqref{c2}, we find that
\begin{align*}
\left[q^{5n+2}\right]\left\{q\frac{f_{10}^5}{f_1f_5^5}\right\}&=\frac{f_5^3f_{10}}{f_1}\left(\frac{G^3(q)}{H^3(q)}+2q\frac{H^2(q)}{G^2(q)}\right)=:\sum_{n=0}^{\infty}A(n)q^n,\\
\left[q^{5n+2}\right]\left\{q\frac{f_5^3f_{10}}{f_1}\right\}&=\frac{f_2f_{25}}{f_1^3}\left(\frac{G^3(q)}{H^3(q)}+2q\frac{H^2(q)}{G^2(q)}\right)=:\sum_{n=0}^{\infty}B(n)q^n,\\
\left[q^{5n+2}\right]\left\{q^2\frac{f_{10}^6}{f_2f_5^2}\right\}&=\frac{f_{50}}{f_1^2}\left(\frac{G^4(q^2)}{H^4(q^2)}-3q^2\frac{H(q^2)}{G(q^2)}\right)=:\sum_{n=0}^{\infty}C(n)q^n.
\end{align*}
Extracting the terms involving $q^{5n+2}$ from both sides of \eqref{before_5_diss}, dividing by $q^2$, replacing $q^5$ by $q$, and then employing the above extraction formulas, we obtain
\begin{align}
\sum_{n=0}^{\infty}\textup{OSOME}(5n+2)q^n
&=\frac{1}{2}\sum_{n=0}^{\infty}A(n)q^n
+\frac{1}{2}\sum_{n=0}^{\infty}B(n)q^n
-2\sum_{n=0}^{\infty}C(n)q^n.
\label{5n+2_dissection}
\end{align}
To prove \eqref{cong-mod5}, now we need to extract the terms involving $q^{8n+5}$ from both sides of the above under modulo 5. We achieve it by extracting $q^{4n+1}$ and $q^{2n+1}$ successively. To that end, with the aid of \eqref{ghf5byf1} and \eqref{overpartition-gf}, we have
\begin{align}\label{An-mod5}
     \sum_{n=0}^{\infty}A(n)q^n
     &=\frac{f_{10}f_5^3}{f_1}\left(\frac{G^3(q)}{H^3(q)}+2q\frac{H^2(q)}{G^2(q)}\right)\notag\\
     &\equiv\frac{f_{10}f_5^3}{f_1}\left(\frac{G^6(q)}{G^3(q)H^3(q)}+2q\frac{H^5(q)G(q)}{G^3(q)H^3(q)}\right)\notag\\
    &\equiv f_1^2f_{10}G(q)G(q^5)+2qf_1^2f_{10}G(q)H(q^5)\notag\\
    &\equiv\varphi(-q)\cdot f_2G(q)\left(f_{10}G(q^5)+2qf_{10}H(q^5)\right)\pmod{5}.
\end{align}

The 4-dissection \eqref{eq:two-dissection-one} can be rewritten as \begin{align}\varphi(-q)&=\varphi(q^4)-2q\psi(q^8).\label{phi-4-dissect}\end{align}

Employing \eqref{watson1}, \eqref{watson2}, and \eqref{phi-4-dissect} we can rewrite \eqref{An-mod5} as
\begin{align*}
 \sum_{n=0}^{\infty}A(n)q^n&\equiv f_8f_{40}\left(\varphi(q^4)-2q\psi(q^8)\right)\left(G\left(q^{16}\right)+qH(-q^4)\right)\\
&\quad\times \left(G\left(q^{80}\right)+2q G(-q^{20})+q^5H(-q^{20})+2q^{16}H(q^{80})\right).
\end{align*}
Extracting the terms involving $q^{4n+1}$ from both sides of the above, we obtain
\begin{align}
    \sum_{n=0}^{\infty}A(4n+1)q^n\notag
    &\equiv f_2f_{10}\Bigl(
 2\varphi(q)G(q^4)G(-q^5)
 +\varphi(q)H(-q)G(q^{20})
 \notag\\
&\qquad
 -2\psi(q^2)G(q^4)G(q^{20})
 +q\varphi(q)G(q^4)H(-q^5)
 \notag\\
&\qquad
 +2q^4\varphi(q)H(-q)H(q^{20}) +q^4\psi(q^2)G(q^4)H(q^{20})\Bigr)\pmod5.\notag
\end{align}
Applying  \eqref{eq:two-dissection-one}, \eqref{watson1}, and
\eqref{watson2} once more in the above, and then extracting the odd exponents of $q$, we find that
\begin{align}
 &\sum_{n=0}^{\infty}A(8n+5)q^n \notag\\
 &\equiv f_1f_{20}\varphi(q^2)G(q^2)G(-q^{10})
 -f_1f_{20}\psi(q^4)G(q^2)G(q^{40})
 +2f_4f_5\psi(q^4)G(-q^2)G(q^{10})
 \notag\\
& \quad-qf_4f_5\varphi(q^2)H(q^8)G(q^{10})-2q^2f_1f_{20}\varphi(q^2)G(q^2)H(-q^{10})
 -q^2f_4f_5\psi(q^4)G(-q^2)H(q^{10})
 \notag\\
&\quad-2q^3f_4f_5\varphi(q^2)H(q^8)H(q^{10})
 -2q^8f_1f_{20}\psi(q^4)G(q^2)H(q^{40})
\notag\\
 &\equiv
f_1f_{20}G(q^{2})
\left(
\varphi(q^2)\left(G(-q^{10})-2q^2H(-q^{10})\right)+2\psi(q^4)\left(2G(q^{40})-q^8H(q^{40})\right)\right)\notag\\
&\quad+f_4f_5\left(G(q^{10})+2q^2H(q^{10})\right)\left(2\psi(q^4)G(-q^2)-q\varphi(q^2)H(q^{8})\right)\pmod5.\label{A(n)}
 \end{align}
In a similar fashion, we find that
 \begin{align}
  &\sum_{n=0}^{\infty}B(8n+5)q^n\notag\\
     &\equiv
f_1f_{20}G(q^2)\bigl(
\varphi(q^{10})\bigl(G(-q^{10})-2q^2H(-q^{10})\bigr)
+2q^2\psi(q^{20})\bigl(2G(q^{40})-q^8H(q^{40})\bigr)\bigr)\notag\\
&\quad+\bigl(G(q^{10})+2q^2H(q^{10})\bigr)
\bigl(f_4f_5\bigl(-q\varphi(q^{10})H(q^8)
+2q^2\psi(q^{20})G(-q^2)\bigr)\notag\\
&\quad-2f_1f_5\psi(q^5)G(q^2)\bigr)\pmod5,\label{B(n)}
 \end{align}
 and
     \begin{align}
&\sum_{n=0}^{\infty}C(8n+5)q^n\notag\\
&\equiv
\frac{2f_2^6f_5^2}{f_1}
\Bigg(
\left(
 \frac{f_1^2f_4^2f_{10}^6}
      {f_2^2f_5^6f_{20}^2}
 +q\frac{f_2^4f_{20}^2}
          {f_4^2f_5^4}
\right)
\notag\\
&\quad\times
\left(
3G(q^4)H^2(-q)
 \left(G(q^{20})+2q^4H(q^{20})\right)
+3G^2(q^4)H(-q)
 \left(2G(-q^5)+qH(-q^5)\right)
\right)
\notag\\
&\quad
-2\frac{f_1f_2f_{10}^3}{f_5^5}
\left(
G^3(q^4)\left(G(q^{20})+2q^4H(q^{20})\right)
+qH^3(-q)
 \bigl(2G(-q^5)+qH(-q^5)\bigr)
\right)
\Bigg)
\notag\\
&\quad+
\frac{2f_1^2f_2^4f_{10}^2}{f_5}
\Bigg(
\frac{f_2^{14}}{f_1^{14}f_4^4}
\Big(
G^3(q^4)\bigl(2G(-q^5)+qH(-q^5)\bigr)
+3G^2(q^4)H(-q)
 \bigl(H(q^{20})\notag\\
&\quad+2q^4H(q^{20})\bigr)
\Big)+4q\frac{f_2^2f_4^4}{f_1^{10}}
\Big(
3G(q^4)H^2(-q)
 \bigl(2G(-q^5)+qH(-q^5)\bigr)\notag\\&\quad
+H^3(-q)
 \bigl(G(q^{20})+2q^4H(q^{20})\bigr)
\Big)
\Bigg)\pmod5.\label{c8nplus5a}
\end{align}
We further simplify the right side of the congruence \eqref{c8nplus5a}. Note from Lemma \ref{lemmaforreduction} that $2G(-q^5)+qH(-q^5)$ can be expressed in terms of $G(q^{20}+2q^4H(q^{20})$. So $G(q^{20}+2q^4H(q^{20})$ appears as a common factor. Simplifying more, omitting some details, we obtain
\begin{align}
    &\sum_{n\geq 0} C(8n+5)q^n\notag\\
    &\equiv-2\dfrac{f_2^7f_{10}^3}{f_5^3}H(-q)G(q^4)\left(G(q^{20})+2q^4H(q^{20})\right)\left(\frac{f_2^2f_{10}^2}{f_1^4f_4^8}-q\frac{f_1^4f_4^8}{f_2^2f_{10}^2}\right)\notag\\
    &\quad\times\left(H(-q)-2G(q^4)\frac{f_1^2f_4^4}{f_2f_{10}}\right)\pmod5.\label{C(n)}
\end{align}

We now simplify \eqref{A(n)} and \eqref{B(n)} together. We have
\begin{align*}
&\frac{1}{2}\sum_{n=0}^{\infty}A(8n+5)q^n
+\frac{1}{2}\sum_{n=0}^{\infty}B(8n+5)q^n\notag\\
    &\equiv\frac{f_1f_{20}G(q^2)}{2}
 \left(
 \Phi\left(G(-q^{10})-2q^2H(-q^{10})\right)
 +2\Psi\left(2G(q^{40})-q^8H(q^{40})\right)
 \right)
 \\
&\quad+\frac{f_5\left(G(q^{10})+2q^2H(q^{10})\right)}{2}
 \left(
 f_4\left(2\Psi G(-q^{2})-q\Phi H(q^8)\right)
 -2f_1\psi(q^5)G(q^2)
 \right)\pmod{5},
\end{align*}
where  $\Phi$ and $\Psi$ are as defined in \eqref{PhiPsi}. By Lemma \ref{lemma=0}, the above reduces to
\begin{align}
&\frac{1}{2}\sum_{n=0}^{\infty}A(8n+5)q^n
+\frac{1}{2}\sum_{n=0}^{\infty}B(8n+5)q^n\notag\\
    &\equiv\frac{f_1f_{20}G(q^2)}{2}
 \left(
 \Phi\bigl(G(-q^{10})-2q^2H(-q^{10})\bigr)
 +2\Psi\bigl(2G(q^{40})-q^8H(q^{40})\bigr)
 \right)\pmod5.\label{G10H10}
\end{align}
Now we express both $G(-q^{10})-2q^2H(-q^{10})$ and $2G(q^{40})-q^8H(q^{40})$ in terms of $G(q^{20})+2q^4H(q^{20})$, one of the common factors on the right side of \eqref{C(n)}. To this end, replacing $q$ by $q^2$ in
\eqref{Identity_1mod5} and then multiplying both sides by $3$ yields
\begin{align}
    G(-q^{10})-2q^2H(-q^{10})
\equiv
\frac{f_2^2f_8^4}{f_4f_{20}}
\left(G(q^{40})+2q^8H(q^{40})\right)\pmod5.\label{eq2a}
\end{align}

Now, using \eqref{GH_relation_3}, we have 
\begin{align*}
\left(
\frac{G(q^{40})+2q^8H(q^{40})}
     {G(q^{20})+2q^4H(q^{20})}
\right)^{2}
&\equiv\frac{f_8f_{20}}{f_4f_{40}}
\equiv\left(\frac{f_4^2}{f_8^2}\right)^2
\pmod5.
\end{align*}
Taking square roots on both sides of the above  and then choosing the appropriate sign by comparing the constant terms, we find that
\begin{align}
G(q^{40})+2q^8H(q^{40})&\equiv\frac{f_4^2}{f_8^2}\left(G(q^{20})+2q^4H(q^{20})\right)\pmod{5}.\label{eq2}
\end{align}
From \eqref{eq2a} and \eqref{eq2}, we have
\begin{align}\label{eq2aa}
    G(-q^{10})-2q^2H(-q^{10})
\equiv
\frac{f_2^2f_4f_8^2}{f_{20}}
\left(G(q^{20})+2q^4H(q^{20})\right)\pmod5.
\end{align}

Again, multiplying both sides of \eqref{eq2} by 2, we have
\begin{align}
2G(q^{40})-q^8H(q^{40})&\equiv2\frac{f_4^2}{f_8^2}\left(G(q^{20})+2q^4H(q^{20})\right)\pmod{5}.\label{eq40-}
\end{align}
Employing \eqref{eq2aa} and \eqref{eq40-} in \eqref{G10H10}, we have
\begin{align}
&\frac{1}{2}\sum_{n=0}^{\infty}A(8n+5)q^n
+\frac{1}{2}\sum_{n=0}^{\infty}B(8n+5)q^n\notag\\
    &\equiv\frac{f_1f_{20}G(q^2)}{2}\left(G(q^{20})+2q^4H(q^{20})\right)
 \left(
 \Phi\frac{f_2^2f_4f_8^2}{f_{20}}
 +2\Psi\frac{f_4^2}{f_8^2}
 \right)\pmod5.\notag
\end{align}
Using the expressions of $\Phi$ and $\Psi$ from \eqref{SON1} and \eqref{SON2} in the above, we find that
\begin{align*}
&\frac{1}{2}\sum_{n=0}^{\infty}A(8n+5)q^n
+\frac{1}{2}\sum_{n=0}^{\infty}B(8n+5)q^n\notag\\
    &\equiv\frac{f_1f_{20}G(q^2)}{2}\left(G(q^{20})+2q^4H(q^{20})\right)
 \left(
 2f_4G(q^2)H(q^8)\frac{f_2^2f_4f_8^2}{f_{20}}
 +2\frac{f_{20}G(q^2)G(q^4)}{H(q^2)G(q^8)}\frac{f_4^2}{f_8^2}
 \right)\notag\\
 &\equiv\left(G(q^{20})+2q^4H(q^{20})\right)
 G^2(q^2)
\left(
f_1f_2^2f_4^2f_8^2G(q^8)
+
\frac{2f_1f_4^2f_{20}^2G(q^4)}{f_8^2H(q^2)G(q^8)}
\right)\pmod5,
\end{align*}
which, by Lemma \ref{lemma2}, reduces to
\begin{align}
   \frac{1}{2}\sum_{n=0}^{\infty}A(8n+5)q^n
+\frac{1}{2}\sum_{n=0}^{\infty}B(8n+5)q^n
\equiv\dfrac{1}{2}f_1f_2^8f_4^2G(q^2) \left(G(q^{20})+2q^4H(q^{20})\right).\label{AB(8n+5)_simplified}
\end{align}

Extracting the terms involving $q^{8n+5}$ from both sides of \eqref{5n+2_dissection} and then using \eqref{C(n)} and \eqref{AB(8n+5)_simplified}, we obtain
\begin{align}
&\sum_{n\geq 0} \textup{OSOME}(40n+27)q^n\notag\\
&\equiv
\left(G(q^{20})+2q^4H(q^{20})\right)
\Big(
\frac{1}{2}G(q^2)f_1f_2^8f_4^2
-2\frac{f_2^9f_{10}^5}{f_1^4f_5^3f_4^8}
G(q^4)H^2(-q)\notag\\
&\quad+4\frac{f_2^8f_{10}^4}{f_1^2f_5^3f_4^4}
G^2(q^4)H(-q)
+2q\frac{f_1^4f_2^5f_4^8f_{10}}{f_5^3}
G(q^4)H^2(-q)\notag\\
&\quad
-4q\frac{f_1^6f_2^4f_4^{12}}{f_5^3}
G^2(q^4)H(-q)
\Big)\pmod5.
\end{align}
Further manipulation of the above by using Lemma \ref{lemma3} yields
\begin{align}
&\sum_{n\geq 0} \textup{OSOME}(40n+27)q^n\notag\\
&\equiv\left(G(q^{20})+2q^4H(q^{20})\right)
\Bigg(\frac{1}{2}G(q^2)f_1f_2^8f_4^2
-2G(q^2)
\left(
\frac{f_2^7f_{10}^6}{f_1^3f_5^4f_4^8}
-q\frac{f_1^5f_2^3f_4^8f_{10}^2}{f_5^4}
\right)\notag\\
&\quad\times
\left(
\psi(q^2)+q\psi(q^{10})
-\frac{f_1^2f_4^4}{f_2f_{10}}
\left(\varphi(q)+\varphi(q^5)\right)
\right)\Bigg)\pmod5.
\label{after-lemma3}
\end{align}

Now,
\begin{align}
\frac{f_1^2f_4^4}{f_2f_{10}}\varphi(q)
&=\frac{f_2^4f_4^2}{f_{10}}
\equiv\frac{f_4^2}{f_2}
=\psi(q^2)\pmod5
\label{theta-cancel}\\\intertext{and}
\frac{f_2^7f_{10}^6}{f_1^3f_5^4f_4^8}
&-q\frac{f_1^5f_2^3f_4^8f_{10}^2}{f_5^4}\equiv
\frac{f_2^3f_{10}^2}{f_4^7f_5^3}
\left(
\frac{f_2^4f_{10}^4}{f_1^3f_5f_4}
-qf_{20}^3
\right)\pmod5.
\label{product-reduction}
\end{align}
Employing \eqref{theta-cancel} and \eqref{product-reduction} in
\eqref{after-lemma3}, we have
\begin{align}
&\sum_{n\geq 0} \textup{OSOME}(40n+27)q^n\notag\\
&\equiv\left(G(q^{20})+2q^4H(q^{20})\right)
\Bigg(\frac{1}{2}G(q^2)f_1f_2^8f_4^2
-2G(q^2)\frac{f_2^3f_{10}^2}{f_4^7f_5^3}
\left(
\frac{f_2^4f_{10}^4}{f_1^3f_5f_4}
-qf_{20}^3
\right)\notag\\
&\quad\times
\left(
q\psi(q^{10})
-\frac{f_1^2f_4^4}{f_2f_{10}}
\varphi(q^5)\right)
\Bigg)\pmod5.
\label{sim_eq_2}
\end{align}

Next, from \cite[p.~509]{ndb-boruah}, we recall that
\begin{align}\label{f_1f_5product}
f_1f_5=\varphi(q^5)\psi(q^2)-q\varphi(q)\psi(q^{10}).
\end{align}
Replacing $q$ by $-q$, we obtain
\begin{align*}
\frac{f_2^3f_{10}^3}{f_1f_4f_5f_{20}}=\frac{f_5^2f_4^2}{f_{10}f_2}
+q\frac{f_1^2f_{20}^2}{f_2f_{10}},
\end{align*}
which by rearrangement implies that
\begin{align}
\frac{f_2^4f_{10}^4}{f_1^3f_5f_4}-qf_{20}^3
=\frac{f_4^2f_5^2f_{20}}{f_1^2}.
\label{BBI3-consequence}
\end{align}
Employing \eqref{BBI3-consequence} in \eqref{sim_eq_2}, we have
\begin{align}
&\sum_{n\geq 0} \textup{OSOME}(40n+27)q^n\notag\\
&\equiv\left(G(q^{20})+2q^4H(q^{20})\right)
\Bigg(\frac{1}{2}G(q^2)f_1f_2^8f_4^2
-2G(q^2)\frac{f_2^3f_{10}^2}{f_4^7f_5^3}\cdot
\frac{f_4^2f_5^2f_{20}}{f_1^2}\notag\\
&\quad\times
\left(
q\psi(q^{10})
-\frac{f_1^2f_4^4}{f_2f_{10}}\varphi(q^5)
\right)
\Bigg)\notag\\
&\equiv\frac{1}{2}G(q^2)f_1f_2^4f_{10}
\left(G(q^{20})+2q^4H(q^{20})\right)\notag\\
&\quad\times\Bigg(\psi(q^2)-\frac{4}{\varphi(-q^5)}\left(q\varphi(-q)\psi(q^{10})
-\varphi^2(-q)\varphi(q^5)\frac{f_4^4}{f_{10}}\right)
\Bigg)\pmod5.\label{lastequation}
\end{align}

However, replacing $q$ by $-q$ in \eqref{f_1f_5product}, we find that
\begin{align*}
\varphi(-q^5)\psi(q^2)+q\varphi(-q)\psi(q^{10})
&=\frac{f_2^3f_{10}^3}{f_1f_4f_5f_{20}}\equiv\varphi^2(-q)\varphi(q^5)\frac{f_4^4}{f_{10}}\pmod5,
\end{align*}
which can be rewritten as 
\begin{align}\label{qphi}
q\varphi(-q)\psi(q^{10})-\varphi^2(-q)\varphi(q^5)\frac{f_4^4}{f_{10}} 
&\equiv-\varphi(-q^5)\psi(q^2)\pmod5.
\end{align}
Using \eqref{qphi} in \eqref{lastequation}, we find that
\begin{align*}
\sum_{n\geq 0} \textup{OSOME}(40n+27)q^n&\equiv\frac{1}{2}f_1f_2^4f_{10}G(q^2)
\left(G(q^{20})+2q^4H(q^{20})\right)
\left(\psi(q^2)+4\psi(q^2)\right)\\
&\equiv0\pmod{5}.
\end{align*}
Therefore,
\begin{align*}
\textup{OSOME}(40n+27)\equiv0\pmod{5},
\end{align*}
which  completes the proof of \ref{cong-mod5}.

\begin{proof}[Proof of Theorem \ref{mod3power}]
We have
\begin{align*}
\sum_{n\geq 0}\textup{OSOME}(n)q^n
&=\frac{1}{8}\left(\frac{f_2}{f_1^2}-\frac{f_1^6}{f_2^3}\right)\\
&\equiv\frac{f_1^6}{f_2^3}-\frac{f_2}{f_1^2}
\equiv\frac{f_3^2}{f_6}-\left(\frac{f_1^2}{f_2}\right)^2\frac{f_6}{f_3^2}\pmod{3}.
\end{align*}

Using \eqref{3disf_1f_2} and extracting the terms involving $q^{3n}$  from both sides of the above and replacing $q^3$ by $q$ we find that 
\begin{align}
\sum_{n\geq 0}\textup{OSOME}(3n)q^n&\equiv\frac{f_1^2}{f_2}-\frac{f_2f_3^4}{f_1^2f_6^2}
\equiv\frac{f_1^2}{f_2}-\left(\frac{f_1^2}{f_2}\right)^5\pmod{3}.\label{OSOME3n}
\end{align}
  Employing \eqref{3disf_1f_2} then extracting the terms involving $q^{3n}$ from both side of the above and replacing $q^3$ by $q$, we obtain
\begin{align}
\sum_{n\geq 0}\textup{OSOME}(9n)q^n
&\equiv\frac{f_3^2}{f_6}-\frac{f_1^2}{f_2}\frac{f_3^4}{f_6^2}
\equiv\left(\frac{f_1^2}{f_2}\right)^3-\left(\frac{f_1^2}{f_2}\right)^7\pmod3\notag\\
&\equiv\frac{f_3^2}{f_6}-\frac{f_3^4}{f_6^2}\left(\frac{f_9^2}{f_{18}}-2q\frac{f_3f_{18}^2}{f_6f_9}\right)\pmod{3}.\label{OSOME9n}
\end{align}
Equating the coefficients of $q^{3n+2}$, we find that
\begin{align}
  \textup{OSOME}(9(3n+2))\equiv0\pmod{3}.\label{cong9n+2}
\end{align}

Again, extracting $q^{3n}$ terms from  both side of \eqref{OSOME9n} and then replacing $q^3$ by $q$, we have
\begin{align}\label{OSOME27n}
\sum_{n\geq 0}\textup{OSOME}(27n)q^n
&\equiv\frac{f_1^2}{f_2}-\left(\frac{f_1^2}{f_2}\right)^5\pmod3.
\end{align}
From \eqref{OSOME3n} and \eqref{OSOME27n}, it follows that
\begin{align*}
\textup{OSOME}\left(27n\right)\equiv\textup{OSOME}\left(3n\right)
\pmod{3}.
\end{align*}  
Therefore, 
\begin{align*}
\textup{OSOME}\left(3^4n\right)\equiv\textup{OSOME}\left(9n\right)
\pmod{3}.
\end{align*}  
By mathematical induction, for any positive integer $k$, we find that
\begin{align*}
\textup{OSOME}\left(3^{2k}n\right)\equiv\textup{OSOME}\left(9n\right)
\pmod{3}.
\end{align*}  
Replacing $n$ by $3n+2$ in the above and then employing \eqref{cong9n+2}, we arrive at
\begin{align*}
    \textup{OSOME}\left(3^{2k+1}n+2\cdot3^{2k}\right)\equiv0\pmod{3},
\end{align*}
which completes the proof of Theorem \ref{mod3power}.
\end{proof}

 \section{Proof of Theorem \ref{mod2power}}
For an arithmetic function $g$, define
\begin{align}
\textup{T}(g,\alpha,\ell,n):={}&
 \ell^3g(2^\alpha\ell^2n)
 +\ell\left(\frac{-2^\alpha n}{\ell}\right)g(2^\alpha n)
 +g\left(\frac{2^\alpha n}{\ell^2}\right)
 -(\ell^3+1)g(2^\alpha n),
\label{T-definition}
\end{align}
where $g(x)=0$ if $x$ is not a nonnegative integer. Shomanov and Garvan \cite{ShomanovGarvan2025} showed that, for every odd prime $\ell$ and all $\alpha,n\geq0$,
\begin{equation}\label{SG-congruence}
\textup{T}(\overline{p},\alpha,\ell,n)\equiv0\pmod{2^{\alpha+12}}.
\end{equation}

We compare this relation with the exact Hecke relation for $c(n)$, where $c(n)$ is as stated in \eqref{c-definition}. To that end, we state and prove two lemmas. 
\begin{lemma}\label{T-c-divisibility}
Let $\ell$ be an odd prime. If $N=2^\alpha n$ and
$$
 \chi:=\left(\frac{-N}{\ell}\right),
$$
then
\begin{align}
\textup{T}(c,\alpha,\ell,n)
 =(\ell^2-1)(\ell^2+1-\ell\chi)c(N)
 -(\ell^2-1)(\ell^2+1)c(N/\ell^2).
\label{T-c-exact}
\end{align}
\end{lemma}

\begin{proof}
By Lemma \ref{c-Hecke-relation},
\begin{align}
    c(\ell^2N)=(\ell+1-\chi)c(N)-\ell c(N/\ell^2).\label{c_ell_2N)}
\end{align}

From \eqref{T-definition}, we  find that
\begin{align*}
\textup{T}(c,\alpha,\ell,n)={}&
 \ell^3c(\ell^2N)
 + \ell\chi c(N)
 +c\left(\frac{N}{\ell^2}\right)
 -(\ell^3+1)c(N),
\end{align*}
 Employing  \eqref{c_ell_2N)} in the above, we readily arrive at \eqref{T-c-exact}.
\end{proof}

\begin{lemma}\label{OSOME-power-two-transfer}If $\ell$ is an odd prime such that $\ell^2\equiv1\pmod{2^{\alpha+12}}$,
then
$$
\textup{T}(\textup{OSOME},\alpha,\ell,n)\equiv0\pmod{2^{\alpha+9}}.
$$
\end{lemma}

\begin{proof}
The closed form of the generating function of $\textup{OSOME}(n)$ stated in Theorem \ref{CF} shows that $\textup{OSOME}(n)$, $\overline{p}(n)$, and $c(n)$  are connected by
\begin{equation}\label{basic-transfer}
 8\textup{OSOME}(n)=\overline{p}(n)-c(n).
\end{equation}
Applying the operator $\textup{T}$ to \eqref{basic-transfer}, we have
$$
 8 \textup{T}(\textup{OSOME},\alpha,\ell,n)
 = \textup{T}(\overline{p},\alpha,\ell,n)- \textup{T}(c,\alpha,\ell,n).
$$
By \eqref{SG-congruence}, the first term on the right side of the equality is divisible by $2^{\alpha+12}$. The hypothesis and Lemma \ref{T-c-divisibility} show that the second term is also divisible by $2^{\alpha+12}$. Hence
$$
 8\textup{T}(\textup{OSOME},\alpha,\ell,n)\equiv0\pmod{2^{\alpha+12}}.
$$
Thus,
$$
 \textup{T}(\textup{OSOME},\alpha,\ell,n)\equiv0\pmod{2^{\alpha+9}}.
$$
\end{proof}

\begin{proof}[Proof of Theorem \ref{mod2power}] Applying Lemma \ref{OSOME-power-two-transfer} with $n$ replaced by
$\ell m$, we obtain
\begin{align}
&\ell^3\textup{OSOME}\left(2^\alpha\ell^3m\right)
+\ell\left(\frac{-2^\alpha\ell m}{\ell}\right)
 \textup{OSOME}\left(2^\alpha\ell m\right)+\textup{OSOME}\left(\frac{2^\alpha m}{\ell}\right)\notag\\
&
\hspace{3cm}-(\ell^3+1)\textup{OSOME}\left(2^\alpha\ell m\right)\equiv0
\pmod{2^{\alpha+9}}.\label{n_replaced_by_lm}
\end{align}
Since
\begin{equation}\label{zero_rel_1}
    \left(\frac{-2^\alpha\ell m}{\ell}\right)=0
\end{equation}
and $\ell\nmid m$, the number $2^\alpha m/\ell$ is not an integer.
Therefore, by the convention in \eqref{T-definition},
\begin{equation}\label{zero_eq_2}
    \textup{OSOME}\left(\frac{2^\alpha m}{\ell}\right)=0.
\end{equation}

Hence, using \eqref{zero_rel_1} and \eqref{zero_eq_2} in \eqref{n_replaced_by_lm}, we find that
$$
\ell^3\textup{OSOME}\left(2^\alpha\ell^3m\right)
-(\ell^3+1)\textup{OSOME}\left(2^\alpha\ell m\right)
\equiv0\pmod{2^{\alpha+9}}.
$$

Furthermore,
$$
 \ell^3+1\equiv0\pmod{2^{\alpha+12}}.
$$
Thus
$$
 \ell^3\textup{OSOME}(2^\alpha\ell^3m)\equiv0\pmod{2^{\alpha+9}}.
$$
Since $\ell$ is odd, $\ell^3$ is invertible modulo $2^{\alpha+9}$, and hence
$$
 \textup{OSOME}(2^\alpha\ell^3m)\equiv0\pmod{2^{\alpha+9}}.
$$
\end{proof}

\section{Proof of Theorems \ref{general-cubic}--\ref{general-square}}\label{sec6}

The two results in this section use the same simple principle. First, an explicit congruence makes the overpartition term in \eqref{basic-transfer} vanish. We then make the theta-cube coefficient vanish, either by the Hecke relation \eqref{c-relation} or by Legendre's three-square theorem.

\begin{proof}[Proof of Theorem \ref{general-cubic}]
For each $m\in\{3,5,7,11\}$, the integer $k_m-2$ is odd. Since
$$
 \ell\equiv-1\pmod m,
$$
we have
$$
 \ell^{k_m-2}\equiv-1\pmod m.
$$
Therefore, Lemma  \ref{Ryan-overpartition-lemma}(1) yields
\begin{equation}\label{pbar-cubic-zero}
 \overline{p}(m\ell^3n)\equiv0\pmod m.
\end{equation}

We next treat the theta-cube coefficient. In \eqref{c-relation}, replace $n$ by $m\ell n$. Since $\ell\mid m\ell n$, the Legendre symbol is zero. Also, because $\ell\nmid mn$, the number $m\ell n/\ell^2$ is not an integer, so the divided-index term is zero. Hence
$$
 c(m\ell^3n)=(\ell+1) c(m\ell n).
$$ 
Applying condition $\ell\equiv-1\pmod m$ in the above, we have
\begin{equation}\label{c-cubic-zero}
 c(m\ell^3n)\equiv0\pmod m.
\end{equation}
From \eqref{basic-transfer}, \eqref{pbar-cubic-zero}, and \eqref{c-cubic-zero}, we obtain
$$
 8\textup{OSOME}(m\ell^3n)\equiv0\pmod m,
$$
which readily implies that
$$
 \textup{OSOME}(m\ell^3n)\equiv0\pmod m.
$$
Thus we complete the proof of Theorem \ref{general-cubic}.
\end{proof}
\begin{proof}[Proof of Theorem \ref{general-square}]
Clearly, the Legendre symbol in \eqref{Ryan-symbol-condition} does not change when $n$ is replaced by $4^jn$. Therefore, Lemma \ref{Ryan-overpartition-lemma}(2) gives
\begin{equation}\label{pbar-square-zero}
 \overline{p}(m4^j\ell^2n)\equiv0\pmod m.
\end{equation}

Next, we show that \begin{equation}\label{c-square-zero} 
 c(m4^j\ell^2n)=0.
\end{equation}
Since $m$ and $\ell$ are odd, we have $m^2\equiv1\pmod8$ and $\ell^2\equiv1\pmod8$. As $n\equiv7m\pmod8$, it follows that
\begin{align*}
 m\ell^2n
 &\equiv m\cdot1\cdot7m\equiv7m^2\equiv7\pmod8.
\end{align*}
Thus, for some integer $b\geq0$,
\begin{align}\label{m4j}
 m4^j\ell^2n=4^j(8b+7).
\end{align}

However, Legendre's well-known three-square theorem \cite[p.24, Theorem1]{Grosswald1985} states that $r_3(N)=0$ if and only if $ N=4^a(8b+7)$ for some $a,b\geq0$. By \eqref{m4j}, it follows that
$$
 r_3(m4^j\ell^2n)=0.
$$
Using this in \eqref{c-r3-relation}, we arrive at \eqref{c-square-zero}.

From \eqref{basic-transfer}, \eqref{pbar-square-zero}, and \eqref{c-square-zero}, we have
$$
 8\textup{OSOME}(m4^j\ell^2n)\equiv0\pmod m,
$$
which yields
$$
 \textup{OSOME}(m4^j\ell^2n)\equiv0\pmod m.
$$
\end{proof}
\section{Concluding Remarks}

\begin{remark}
Ryan, Sirolli, Villegas-Morales, and Zheng
\cite{RyanSirolliVillegasZheng2024} obtained a finite list of pairs
of primes $(m,\ell)$, with $m\in\{13,17,19\}$, for which
$$
\overline{p}(m\ell^3n)\equiv0\pmod m
$$
whenever $\ell\nmid n$. They have listed the pairs in Table 1 of \cite{RyanSirolliVillegasZheng2024}. We notice that each of the following pairs from that table satisfies
$\ell\equiv-1\pmod m$:
$$
\begin{aligned}
m=13:\quad&\ell=1871,1949,3301,4289,\\
m=17:\quad&\ell=2039,2719,3331,4079,\\
m=19:\quad&\ell=151,2659,3989.
\end{aligned}
$$
Consequently, for each pair in the above list, for every positive integer
$n$ prime to $\ell$, one can show that
\begin{align}\label{note-finite-cubic}
\textup{OSOME}(m\ell^3n)\equiv0\pmod m.
\end{align}

Ryan, Sirolli, Villegas-Morales, and Zheng
\cite{RyanSirolliVillegasZheng2024} also found a finite list of
triples $(m,\ell,\varepsilon_{m,\ell})$, where
$m\in\{13,17,19\}$ and $\varepsilon_{m,\ell}\in\{-1,1\}$ for which
$$
\overline{p}(m\ell^2n)\equiv0\pmod m
$$
for every positive integer $n$ prime to $\ell$ and 
$$
\left(\frac{(-1)^{(m-3)/2}n}{\ell}\right)
=\varepsilon_{m,\ell}.
$$
The triples listed in their table \cite[Table $2$]{RyanSirolliVillegasZheng2024} are given below:
$$
\begin{aligned}
m=13:\quad&
(\ell,\varepsilon_{13,\ell})
=(431,1),(2459,1),(4513,1),(4799,1),\\
m=17:\quad&
(\ell,\varepsilon_{17,\ell})
=(167,1),(541,1),(911,-1),(1013,-1),(1153,1),(1867,1),\\
&\hspace{20mm}(1931,-1),
(2543,-1),(2683,1),(2887,1),(3019,-1),\\
&\hspace{20mm}
(3023,1),(3329,1),(4243,-1),(4651,-1),\\
m=19:\quad&
(\ell,\varepsilon_{19,\ell})=(2207,-1).
\end{aligned}
$$
These congruences can be transferred into congruences for $\textup{OSOME}(n)$ after
imposing an additional condition. More precisely, if
$$
\ell\nmid n,\qquad
\left(\frac{(-1)^{(m-3)/2}n}{\ell}\right)
=\varepsilon_{m,\ell},
\qquad\text{and}\qquad
n\equiv7m\pmod8,
$$
then, for every integer $j\geq0$, one has
\begin{align}\label{note-finite-square}
\textup{OSOME}\left(m4^j\ell^2n\right)\equiv0\pmod m.
\end{align}
\end{remark}

The proofs of \eqref{note-finite-cubic} and \eqref{note-finite-square} are similar in nature to those given in Section \ref{sec6}. Therefore, we omit the details.

\begin{remark}The functions $\textup{OSOME}_o(n)$ and $\textup{OSOME}_e(n)$ also satisfy interesting congruences, which can be proved using the methods discussed in Section \ref{sec4}. A sample of congrunces are given below:
$$\textup{OSOME}_o(4n + 3) \equiv 0 \pmod{2},$$ 
$$\textup{OSOME}_o(6n + 5) \equiv 0 \pmod{2},$$
$$\textup{OSOME}_o(72n+11)\equiv0\pmod3,$$
$$\textup{OSOME}_o(96n+76)\equiv0\pmod3,$$
$$\textup{OSOME}_e(72n + 59) \equiv 0 \pmod{3},$$ 
$$\textup{OSOME}_e(96n + 12) \equiv 0 \pmod{3}, $$
$$\textup{OSOME}_e(96n + 76) \equiv 0 \pmod{3},$$
$$\textup{OSOME}_o(8n + 7) \equiv 0 \pmod{4},$$
$$\textup{OSOME}_o(16n + 14) \equiv 0 \pmod{4},$$
$$\textup{OSOME}_o(40n+31)\equiv0\pmod5,$$
$$\textup{OSOME}_e(40n+7)\equiv0\pmod5.$$

Furthermore, based on numerical evidence, we propose the following conjectures modulo $7$. 

\begin{conjecture}
    For every nonnegative integer $n$, 
\begin{align*}\textup{OSOME}(56n+43)&\equiv0\pmod7,\\
\textup{OSOME}_o(56n + 11) &\equiv 0 \pmod{7},\\
\textup{OSOME}_o(56n + 23) &\equiv 0 \pmod{7},\\
\textup{OSOME}_o(56n + 47) &\equiv 0 \pmod{7},\\
\textup{OSOME}_e(56n+51)&\equiv0\pmod7.
\end{align*}
\end{conjecture}
\end{remark}

\section{Acknowledgments}
The second author was partially supported by the University Grants Commission, Government of India, under
the UGC-JRF scheme (Ref. No. 221610056019). The author thanks the funding agency. 

\end{document}